\documentclass[10pt]{amsart}
\usepackage{amsfonts,amssymb,amsmath, forest}
\usepackage{amssymb, amsmath}
\usepackage{float,epsfig}
\usepackage{xcolor}
\usepackage{ mathrsfs }
\usepackage{graphicx}
\usepackage{hyperref}
\usepackage{mathabx}
\usepackage[left=2.5cm,top=1 cm,right=2.5cm,nofoot]{geometry}

\usepackage{cleveref}
\usepackage[mathlines]{lineno}

\newtheorem{thm}{Theorem}[section]

\newtheorem{deff}[thm]{Definition}

\newtheorem{prop}[thm]{Proposition}

\newtheorem{remark}{Remark}

\newcommand{\mc}{\mathcal}

\renewcommand{\ss}{\subseteq}
\newcommand{\mf}{\mathfrak}
\newcommand{\ra}{\rightarrow}
\newcommand{\msc}{\mathscr}
\newcommand{\ol}{\overline}

\newcommand{\mfrak}{\mathfrak}
\newcommand{\lan}{\langle}
\newcommand{\ran}{\rangle}

\def \l2x{L^2(X;\mc H)}

\def \Pf{{\mbox{{\bf Pf}}}}
\def \HS{{\mc B_2}}
\def \Ell{{\ell^2(\mathbb Z^r,\mc B_2(L^2(\mathbb R^d)))}}
\DeclareMathOperator*{\Span}{span}

\newcommand{\blue}{\textcolor{blue}}

\theoremstyle{definition}

\theoremstyle{remark}

\numberwithin{equation}{section}

\usepackage{scalerel,stackengine}
\stackMath
\newcommand\reallywidehat[1]{%
	\savestack{\tmpbox}{\stretchto{%
			\scaleto{%
				\scalerel*[\widthof{\ensuremath{#1}}]{\kern.1pt\mathchar"0362\kern.1pt}%
				{\rule{0ex}{\textheight}}
			}{\textheight}%
		}{2.4ex}}%
	\stackon[-6.9pt]{#1}{\tmpbox}%
}
\begin{document}

\title[Reproducing FORMULAS  ON NILPOTENT LIE GROUPS]{REPRODUCING FORMULAS ASSOCIATED TO\\
	translation GENERATED SYSTEMS ON NILPOTENT LIE GROUPS}
\author{Sudipta Sarkar}
\address{Department of Mathematics,
	Indian Institute of Technology Indore,
	Simrol, Khandwa Road,
	Indore-453 552}
\email{sudipta.math7@gmail.com, nirajshukla@iiti.ac.in}

\thanks{Research of S. Sarkar and N. K. Shukla was supported by   research grant from CSIR, New Delhi [09/1022(0037)/2017-EMR-I] and NBHM-DAE [02011/19/2018-NBHM(R.P.)/R\&D II/14723], respectively.}

\author{Niraj K. Shukla}




\begin{abstract}   Let $G$ be a connected, simply connected,  nilpotent Lie group whose irreducible unitary representations are square-integrable modulo the center. We obtain characterization results for reproducing formulas associated with the left translation generated systems in $ L^2(G)$. Unlike the previous study of discrete frames on the nilpotent Lie groups, the current research occurs within the set up of continuous frames, which means the resulting reproducing formulas are given in terms of integral representations instead of discrete sums. As a consequence of our results for the Heisenberg group, a reproducing formula associated with the orthonormal Gabor systems of $L^2(\mathbb R^d)$ is obtained. 
\end{abstract}


\maketitle

%
%

\section{{\bf Introduction }}\label{intro}
 One of the most attractive areas of research in Harmonic analysis is to reproduce a function from a given set of functions via a formula known as \textit{reproducing formula}.   Such formulas are studied  in various platforms like Gabor, wavelet, and shift generated systems  in the Euclidean and locally compact abelian  (LCA) group set up due to its wide use in various areas: time-frequency analysis and mathematical physics,    quantum mechanics,   quantum optics, etc., (see,  \cite{bownik2007SMI, christensen2004oblique, gabardo2009properties, gabardoDual2004uniqueness,heil2009duals,hemmat2007uniqueness}).
 In general, the aim is to find $\{\psi_i\}_{i\in I}$,  for the given system $\{\varphi_i\}_{i\in I}$, where $I$ is countable such that the below reproducing formula 
$$
\quad f=\sum_{i\in I, k\in \mathbb Z^n} \langle f, \psi_{i}(\cdot-k)\rangle \varphi_i(\cdot-k), \quad  \mbox{for all} \  f \in\ol{ \mbox{span}}\{\varphi_i(\cdot-k) \}_{i, k}
$$
holds  true  provided both the systems $\{\varphi_i(\cdot-k)\}_{i, k}$ and $\{\psi_i(\cdot-k)\}_{i, k}$ are Bessel. The system $\{\psi_i(\cdot-k)\}_{i, k}$  is known as \textit{alternate dual} to   $\{\varphi_i(\cdot-k)\}_{i, k}$,  where $\{\varphi_i(\cdot-k)\}_{i, k}$ is a frame sequence satisfying the above formula.
 It can be noted that the above stable decomposition of $f$ allows the flexibility of choosing different type of duals   of a   frame $\{\varphi_i(\cdot-k) \}_{i, k}$. When the frame sequence $\{\varphi_i(\cdot-k) \}_{i, k}$ becomes Riesz basis, the choice of $\psi_i$ is unique and the system $\{\psi_i(\cdot-k)\}_{i, k}$  is   \textit{biorthogonal dual} to   $\{\varphi_i(\cdot-k)\}_{i, k}$. In general\blue{,}    the   researchers  consider  $\ol{ \mbox{span}}\{\varphi_i(\cdot-k)\}_{i, k}=\ol{ \mbox{span}}\{\psi_i(\cdot-k)\}_{i, k}$  but sometimes  choosing $\psi_i$'s outside the space $\ol{ \mbox{span}}\{\varphi_i(\cdot-k)\}_{i, k}$    provides better localization properties in both the time and the  frequency domains. Such  kind of requirements motivate researchers to define various  duals of a frame, like, oblique dual, type-I and  type-II duals, etc. We refer   \cite{  gabardo2009properties, han2000frames, heil2009duals, hemmat2007uniqueness} for more details.

Our main goal is to describe the reproducing formulas associated with the left translation generated  continuous frame systems in   $L^2 (G)$.
 In particular, we assume $G$  to be a  connected, simply connected, nilpotent Lie group with Lie  algebra $\mathfrak{g}$ and   $Z$ be  the  center of $G$.  Then $G$ is an \textit{$SI/Z$ group} if almost all of its irreducible representations are square-integrable (SI) modulo the center $Z$.
An irreducible representation $\pi$ of $G$ is called \textit{square integrable modulo the center} if  it satisfies the condition 
\begin{align*}
	\int_{G/Z}|\langle \pi({g}) u,v\rangle|^2d \dot{g}<\infty, \ \mbox{for all}\ u, v.
\end{align*}
The motive of the current study has two folds: First it touches the non-abelian set up of the nilpotent Lie group which is considered to be the high degree of non-abelian structure, and secondly, it captures the non-discrete translation. Consequently, the theory is valid for the $d$-dimensional Heisenberg group $\mathbb H^d$, a 2-step nilpotent Lie group. Indeed,  our work is the continuation of a chain of work of  Currey et al. \cite{currey2014characterization} for $L^2(G)$, and  Barbieri et al. \cite{bar2014bracket} for   $L^2 (\mathbb H^d)$.

 We briefly start by describing left translation generated  systems in $L^2 (G)$ as follows: For a sequence  of functions  $\msc A =\{\varphi_k : k\in  I\}$ in $L^2( G)$ and   a subset $\Lambda$ of $G$,   we   define \textit{$\Lambda$-translation generated ($\Lambda$-TG) system}  $\mc E^{\Lambda} (\msc A)$ and its associated  \textit{$\Lambda$-translation invariant   ($\Lambda$-TI)  space} $\mc S^{\Lambda} (\msc A)$   by the action of $\Lambda$ as follows: 
\begin{align}\label{TIsystem}
\mc E^{\Lambda} (\msc A):=\{L_\lambda \varphi :    \lambda \in  \Lambda,\varphi \in \msc A\} , \ \mbox{and} \ \mc S^{\Lambda}(\msc A) := \ol{\mbox{span}} \ \mc E^{\Lambda}(\msc A),
\end{align}
where for each $\lambda \in G$, the \textit{left translation operator} $L_\lambda$ on $L^2 (G)$   is defined by $L_{\lambda} f (x)= f (\lambda^{-1}x),$ for $x \in G$. A   closed subspace $V$ of $L^2(G)$ is  \textit{$\Lambda$-TI  space}  if $f\in V$ implies $L_\lambda f\in V$ for all $\lambda\in \Lambda$.  We denote $\mc E^{\Lambda} (\msc A)=\mc E^{\Lambda} (\varphi)$ and $\mc S^{\Lambda} (\msc A)=\mc S^{\Lambda} (\varphi)$ for $\msc A=\{\varphi\}$. For     an  integer lattice $\Lambda_0$ in the center of $G$, we   particularly    consider  $\Lambda=\{\lambda_1 \lambda_0: \lambda_i \in \Lambda_i, \ i=0, 1\}$,  where    $\Lambda_1$ is a subset (not necessarily discrete) of   $G$.

Due to the wide use of continuous frames in various areas, like harmonic analysis,  mathematical physics,    quantum mechanics,   quantum optics, etc.,   we discuss the reproducing formulas for $\mc E^{\Lambda}(\msc A)$ associated with the general set $\Lambda_1$ {need not be discrete}. The  current study  provides a compendious   study of duals for a continuous frame $\mc E^{\Lambda}(\msc A)$ of $\mc S^{\Lambda}(\msc A)$.     
In the continuous setup, we provide point-wise characterizations of alternate duals and their subcategories like oblique, type I, and type II duals for the continuous frame $\mc E^\Lambda(\msc A)$ and show that the global property of duals can be transmitted into locally.   Unlike the  Euclidean case  \cite{hemmat2007uniqueness}, it is worthwhile to mention that the  point-wise characterization results  depends  on $\Lambda_1$.  
 We can get results for discrete frames by sampling the continuous frames, which keep wide applications due to their computational simplicity.  Further for the discrete set $\Lambda_1$, we discuss reproducing formula associated to the singly generated systems $\mc E^{\Lambda} (\varphi)$ and $\mc E^{\Lambda} (\psi)$ having biorthogonal property, where $\varphi, \psi \in L^2 (G)$.  At this juncture, we point out that   the technique used in the Euclidean and LCA groups is restrained since   the dual space of the nilpotent Lie groups replaces the frequency domain, and the Plancherel transform of a function is operator-valued.

To illustrate the current work for the Heisenberg group $\mathbb H^d$, we first note that $\mathbb H^d$ is an $SI/Z$ nilpotent Lie group identified with $\mathbb  R^d \times \mathbb R^d \times \mathbb R$.  The \textit{$d$-dimensional Heisenberg group},  denoted by $\mathbb  H^d$, is an example of $SI/Z$ group. The group $\mathbb H^d$ can be   identified with $\mathbb R^d\times \mathbb R^d\times \mathbb R$ 
under the group operation 
$(x,y,w).(x',y',w')=(x+x',y+y',w+w'+x\cdot y)$, $x, x', y, y' \in \mathbb R^d$, $w, w' \in \mathbb R$,  where  ``$\cdot$'' stands for standard dot product.  The results which has been derived  so far can be converted for $L^2 (\mathbb H^d)$ using the fiberization map associated with the   Schr\"odinger representations on $L^2 (\mathbb R^d)$ in a natural way.  In this case  we can choose   $\Lambda_1, \Lambda_0$ of the form $\Lambda_1=
	\Gamma_1\times \Gamma_2\times \{0\}$ and $\Lambda_0=\{(0,0)\times m \mathbb Z\}$, where $\Gamma_1$, $\Gamma_2$ are additive subgroups of $\mathbb R^d$ and $m \in \mathbb N$.
 As a consequence of our results for the Heisenberg group, a reproducing formula associated with the orthonormal Gabor systems of $L^2(\mathbb R^d)$ is obtained.


%

The paper is organized as follows: In Section \ref{S:Pre}  we provide a brief discussion about the  Plancherel transform for the $SI/Z$ nilpotent Lie group. Section \ref{S:non-discrete} is devoted to constructing reproducing formulas associated with the continuous frames generated by the non-discrete translation of multiple functions using the range function. Employing the Plancherel transform followed by periodization, we discuss the reproducing formula for the single generated discrete  system  $\mc E^{\Lambda} (\varphi)$ and $\mc E^{\Lambda} (\psi)$  having biorthogonal property in Section \ref{S:Rieszbasis}.

\section{Plancherel transform for SI/Z nilpotent Lie group}\label{S:Pre}
Let $G$ be a connected, simply connected, nilpotent Lie group with Lie  algebra $\mathfrak{g}$. We identify 
$G$ with $\mfrak g\cong \mathbb R^n$ due to the analytic diffeomorphism of the exponential map $\exp:\mathfrak g\rightarrow G$, where $n=\mbox{dim}~{\mfrak g}$. To choose a basis for the Lie  algebra $\mathfrak{g}$, we consider the Jordon-H\"older series $(0)\ss \mfrak g_1\ss \mfrak g_2\ss\dots \ss \mfrak g_n=\mathfrak g$ of ideals of $\mfrak g$   such that dim~$\mfrak g_j=j$ for $j=0,1,\dots, n$ satisfying  ad$(X)\mfrak g_j\ss \mfrak g_{j-1}$  for $j=1,\dots, n$, for all $X\in \mfrak g$, where for $X, Y\in \mfrak g$, ad$(X)(Y)=[X, Y]$, the Lie bracket of $X$ and $Y$.  Now   we pick $X_j\in \mathfrak g_j\backslash \mathfrak g_{j-1}$  for each $j=1,2,\dots, n$, such that the collection $\{X_1, X_2, \dots, X_n\}$ is a Jordan-H\"older basis.  Then  the map $\mathbb R^n\longrightarrow \mathfrak g\longrightarrow G$ defined by $(x_1,x_2,\dots,x_n)\mapsto \sum_{j=1}^n x_jX_j\mapsto \exp(\sum_{j=1}^n x_jX_j)$ is a diffeomorphism, and  hence the Lebesgue measure on $\mathbb R^n$ can be realized as a Haar measure on $G$ \cite{corwin2004representations}.

Note that the center $\mathfrak z$ of the Lie algebra $\mathfrak{g}$ is non-trivial, and it maps to the center $Z:=\exp{\mathfrak z}$ of $G$.
The Lie group $G$ acts on $\mathfrak g$  and $\mathfrak g^*$ by the \textit{adjoint action} $\exp(Ad(x)X):=x\exp(X)x^{-1}$ and   \textit{co-adjoint action} $(Ad^*(x)\ell)(X) =\ell(Ad(x^{-1})X)$, respectively,  for  $x\in G, X\in \mathfrak g,$ and $ \ell\in \mathfrak g^*$.     The  $\mathfrak g^*$ denotes  the  vector space of all real-valued linear functionals on $\mathfrak g$.  For  $\ell\in \mathfrak g^*$ the   \textit{stabilizer} $R_{\ell} =\{x\in G:(Ad^*x)\ell=\ell\}$ is a Lie group   with the associated Lie algebra $r_{\ell}:=\{X\in \mathfrak{g}: \ell[Y, X]=0, \mbox{for all} \ Y \in \mathfrak g\}$. 


Our aim is to discuss   Kirilov Theory \cite{corwin2004representations}   to define the Plancherel transform for $SI/Z$ group. Given any $\ell\in \mathfrak{g^*}$ there exists a subalgebra~ $\mfrak{h}_{\ell}$ (known as \textit{polarizing} or \textit{maximal subordinate  subalgebra}) of $\mathfrak{g}$ which is maximal with respect to the property
$\ell[\mathfrak{h}_{\ell},\mfrak{h}_{\ell}]=0$. Then the map  $\mathcal{X}_{\ell}:\exp(\mfrak{h}_{\ell})\rightarrow \mathbb T$ defined  by
$\mathcal X_{\ell}(\exp X)=e^{2\pi i \ell(X)}$,  $X\in \mfrak{h}_{\ell}$ is a character on $\exp(\mfrak{h}_{\ell})$, and hence the representations induced from $\mathcal X_{l}$, $\pi_{\ell} : =\mbox{ind}_{\exp{\mathfrak h_{\ell}}}^G$$\mathcal{X}_{\ell}$, have the following properties:
\begin{itemize}
	\item[(i)] $\pi_{\ell}$ is an irreducible unitary representation of $G$.
	\item[(ii)]  Suppose $\mfrak{h'}_{\ell}$ is another subalgebra which is maximal with respect to the property $\ell[{\mathfrak {h}}'_{\ell}, {\mfrak{h}}'_{\ell}]=0$, then $\mbox{ind}_{\exp{\mfrak{h}_{\ell}}}^G\mathcal X_{\ell}\cong\mbox{ind}_{\exp(\mfrak h'_{\ell})}^G\mathcal X_{\ell'}$. 
	\item[(iii)] $\pi_{\ell_1}\cong \pi_{\ell_2}$ if and only if $\ell_1$ and $\ell_2$ lie in the same co-adjoint orbit.
	\item[(iv)] Suppose $\pi$ is a irreducible unitary representation of $G$, then there exists $\ell\in \mfrak{g^*}$ such that $\pi\cong \pi_{\ell}$.
\end{itemize}
Therefore there exists  a   bijection 
$\mf t^*:\mfrak g^*/Ad^*(G)\rightarrow \widehat{G}$ which is also  a Borel isomorphism, where $\widehat{G}$ is the collection of all irreducible unitary representations of $G$.

For an irreducible representation $\pi\in \widehat G$, let $O_{\pi}$ denote as coadjoint orbit corresponding to the equivalence class of $\pi$. Then the orbital characterization for the SI/Z representation is: 
$\pi$ is square integrable modulo the center if and only if  for $\ell\in O_\pi$, $r_{\ell}=\mathfrak z$ and $O_\pi=\ell+\mathfrak z^{\perp}$. If SI/Z $\neq \phi$,  then SI/Z=$\widehat{G}_{max}$, where the Borel subset  $\widehat {G}_{max}\ss \widehat G$ corresponds  to coadjoint orbits of maximal dimension which is co-null for Plancherel measure class. Hence when $G$ is an SI/Z group,  $\widehat G_{max}$ is parameterized by a subset of $\mathfrak z^*$. If  $\pi \in \widehat G_{\max}$, then  dim $O_\pi=n-\mbox{dim} \  \mathfrak z$, since $O_\pi$ is symplectic manifold, it is of even dimension, say,  $\dim O_\pi=2d$.  By Schurs' Lemma the restriction of $\pi$ to $Z$ is a  character and hence  it is a unique element $\sigma=\sigma_\pi \in \mathfrak z^*$ (say) and $\pi(z)=e^{2\pi i\langle \sigma, log z\rangle}I$, where $I $ is the identity operator. It shows that $O_\pi=\{l\in \mathfrak g^*:l|_{\mathfrak z}=\sigma\}$ and $\pi\mapsto \sigma_\pi$ is injective.

Let $G$ be  an SI/Z group    and $\mc W=\{\sigma \in \mathfrak{z^*}:{\bf Pf}(\sigma)\neq 0\}$
be a cross section for the coadjoint orbits of maximal dimension, where  the \textit{Pfaffian determinant}   ${\bf Pf}: \mathfrak z^*\longrightarrow \mathbb R$ is given by  $\ell \mapsto \sqrt{\big| \det(\ell[X_i, X_j] )_{i,j=r\dots n}\big|}$.  Then  for a fixed $\sigma \in \mc W$,  $p(\sigma)=\sum_{j=1}^d \mathfrak g_j(\sigma\big|_{\mathfrak g_j})$ is a  maximal subordinate subalgebra for $\sigma$ and the corresponding induced representation $\pi_\sigma$ is realized naturally in $L^2 (\mathbb R^d)$, where $n=r+2d$ for some $d$. For each $\varphi \in L^1(G)\cap L^2(G)$, the Fourier transform of $\varphi$  given  by  
$$\widehat{\varphi}(\sigma)=\int_{G} \varphi(x)\pi_{\sigma}(x)dx, \quad \sigma \in \mc W$$
defines a trace-class operator on $L^2(\mathbb R^d)$, with the inner product $\langle A, B\rangle_{ \HS}=tr(B^*A)$. This space is denoted by $\HS(L^2(\mathbb R^d))$.  When $d\sigma$ is suitable normalized,
$$ \|\varphi\|^2 =\int_{\mc W}\|\widehat{\varphi}(\sigma)\|^2_{\HS(L^2(\mathbb R^d))}|{\bf Pf}(\sigma)| d(\sigma).$$
The Fourier transform can be extended unitarily as $\mc F$ - the \textit{Plancherel transform},
\begin{align*}
	\mc F:L^2(G) & \rightarrow L^2(\mathfrak z^*,\HS(L^2(\mathbb R^d)), |{\bf Pf}(\sigma)|d\sigma)) ,   \quad  \mc Ff=\widehat f.
\end{align*}
Note that the Plancherel transform  $\mc F$ satisfies the relation 
$$\mc F (L_{\lambda} f )(\sigma)=\pi_{\sigma}(\lambda) \mc F f (\sigma), \  \mbox{for} \ \lambda \in G,  \ a.e. \  \sigma \in \mathfrak z^*,  \ \mbox{and} \  f \in L^2(G),$$
where   the left translation operator $L_\lambda$ on $L^2 (G)$  is given by $L_{\lambda} f (x)= f (\lambda^{-1}x)$.
 
  Applying the periodization on $\mc F$ there is a unitary map  (see, \cite{currey2014characterization,sudipta2022extrainvariant}), $\mathscr F:L^2(G)\rightarrow L^2( \mathbb T^r;\ell^2(\mathbb Z^r, \HS(L^2(\mathbb R^d))))$ 
	\begin{equation}\label{T:fiberization} 
		\mathscr Ff(\alpha)(m)=\mc Ff(\alpha+m)|\Pf(\alpha+m)|^{1/2}, \ \mbox{for all } f\in L^2(G), \ m \in \mathbb Z^r, \ \mbox{a.e.}\  \alpha \in \mathbb T^r.
	\end{equation}
  Moreover,  $\msc F$ satisfies the intertwining property of left translation by the action of  $\Lambda$   with the representation $\tilde \pi_{}$ as follows:
	\begin{equation}\label{eq:intertwining}
		\msc F(L_{\lambda}f)(\alpha)=e^{2\pi i\langle \alpha, \lambda_0\rangle}\tilde{\pi}(\lambda_1)\msc Ff(\alpha), \quad  \mbox{a.e.}\  \alpha \in \mathbb T^r,
	\end{equation}
	where    for $\lambda_0 \in \Lambda_0$ and $\lambda_1 \in \Lambda_1$,  we write $\lambda=\lambda_1\lambda_0$, and for $g \in G$ and  $h\in  L^2( \mathbb T^r;\ell^2(\mathbb Z^r, \HS(L^2(\mathbb R^d))))$,      $\tilde{\pi}(g) h(\alpha)=\tilde{\pi}_{\alpha}(g)h(\alpha)$, a.e. $\alpha \in \mathbb T^r$. The representation  $\tilde \pi_{\alpha}(g)$ is defined on $\Ell$ by $\tilde{\pi}_{\alpha}(g) z(m)=\pi_{\alpha+m}(g)\degree z(m)$ for $m \in \mathbb Z^r$, where $(z(m)) \in \Ell$, $\degree$ denotes the composition of operators in $\HS(L^2(\mathbb R^d))$ and $\pi_{\alpha+m}(g)$ is the Hilbert-Schmidt operator defined on $L^2(\mathbb R^d)$.

\section{Reproducing formulas associated to continuous frames} \label{S:non-discrete}
Through out the section we assume  $G$ to be a connected, simply connected, nilpotent Lie group with Lie  algebra $\mathfrak{g}$. The  Haar measure on $G$ can be realised as a Lebesgue measure on $\mathbb R^d$. Further assume an arbitrary measurable subset $$\Lambda_1\ss\exp\mathbb RX_{r+1}\dots \exp \mathbb RX_{n} \ \mbox{(not necessarily discrete)} $$ and the  integer lattice $\Lambda_0=\exp \mathbb ZX_1\dots\exp \mathbb ZX_r $ in $G$. For the   
 countable collection  of functions
$\msc A=\{\varphi_k: k\in I\}$ and $\msc A'=\{\psi_k: k\in I\}$ in $L^2(G)$, we recall the  $\Lambda$-translation generated systems  $$\mc E^\Lambda(\msc A)=\{L_\lambda\varphi_k: \lambda\in \Lambda, k\in I \}\  \mbox{and} \  \mc E^\Lambda(\msc A')=\{L_\lambda\psi_k: \lambda\in \Lambda, k\in I \},$$ where $\Lambda=\Lambda_1\Lambda_0=\{\lambda_1\lambda_0:\lambda_1\in \Lambda_1, \lambda_0\in \Lambda_0\}$. The set $\Lambda=\Lambda_1\Lambda_0$ is measurable.

In the present section our aim is to develop theory for the reproducing formulas in the continuous setup having multiple generators.  
Two Bessel families $\mc E^\Lambda(\msc A)$ and $\mc E^\Lambda(\msc A')$ are said to be \textit{dual frames} if  
{\small$$
	<f, g>=\sum_{k \in I}\int_\Lambda <f, L_\lambda\psi_k><L_\lambda\varphi_k, g> d\lambda, \ \mbox{for all} \ f, g \in L^2 (G). 
	$$}
In this case the reproducing formula (in weak sense)  is 
{\small\begin{equation}\label{rformula}
		f =\sum_{k \in I}\int_\Lambda <f, L_\lambda\psi_k>L_\lambda\varphi_k d\lambda, \ \mbox{for all} \ f  \in L^2 (G),
\end{equation}}
and  both the systems $\mc E^\Lambda(\msc A)$ and $\mc E^\Lambda(\msc A')$ will become a continuous frame for $L^2(G)$ in view of Cauchy-Schwarz inequality.  By a \textit{Bessel family} $\mc E^\Lambda(\msc A)$, we mean  there is a  $B>0$ such that 
 $$
	\sum_{k \in I} \int_{\Lambda} |<f, L_{\lambda} \varphi_k>|^2  d \lambda \leq B \|f\|^2, \ \mbox{for all} \ f  \in L^2 (G).
	$$ 
In addition it is  a \textit{continuous frame} for  $L^2 (G)$  if  there exists a $0<A\leq B<\infty$ such that $A \|f\|^2$ is a lower bound of the above expression.
If both the lower and upper inequality  holds for all $f\in \mc S^\Lambda(\msc A)$, then $\mc E^\Lambda(\msc A)$ is  a frame for $\mc S^\Lambda(\msc A)$.
In the  case of  countable set $\Lambda$ and counting measure $\mu_\Lambda$, the expression (\ref{rformula}) becomes \begin{equation*}
		f =\sum_{k \in I}\sum_{\lambda\in \Lambda} <f, L_\lambda\psi_k>L_\lambda\varphi_k, \ \mbox{for all} \ f  \in L^2 (G).
\end{equation*} 
The reproducing formula (\ref{rformula}) might hold under much weaker restrictions   on $\mc E^\Lambda(\msc A)$ and $\mc E^\Lambda(\msc A')$.  Assume $\mc E^\Lambda(\msc A)$ to be a continuous frame for   $\mc S^\Lambda(\msc A)$  and  $\mc E^\Lambda(\msc A')$  is Bessel so that the   reproducing formula (\ref{rformula})  holds for all $f \in \mc S^\Lambda(\msc A)$ then
the system	$\mc E^\Lambda(\msc A')$ is said to be  an \textit{alternate dual} of $\mc E^\Lambda(\msc A)$. In this case $\mc E^\Lambda(\msc A')$ may not be a continuous frame for $\mc S^\Lambda(\msc A')$ and $\mc E^\Lambda(\msc A)$ need not be an alternate dual of $\mc E^\Lambda(\msc A')$. If it is so,  then the   alternate dual $\mc E^\Lambda(\msc A')$ is  an \textit{oblique dual} of $\mc E^\Lambda(\msc A)$.  To get more flexibility for choosing of alternate duals, we further categorize it into two more categories based on the range of analysis operator viz., type-I and type-II duals.  $\mc E^{\Lambda}(\msc A')$ is a \textit{type-I (type-II) dual} of $\mc E^{\Lambda}(\msc A)$ if it is an alternate dual of $\mc E^{\Lambda}(\msc A)$  and 
$\mbox{range} (T_{{\mc E^{\Lambda}(\msc A')}}^*)\subset \mbox{range} (T_{\mc E^{\Lambda}(\msc A)}^*)$ ($\mbox{range} (T_{\mathcal E^{\Lambda}(\msc A')})\subset \mbox{range} (T_{\mc E^{\Lambda}(\msc A)})$), where  the analysis $T_{\mc E^\Lambda(\msc A)}$ and   synthesis   $T^*_{\mc E^\Lambda(\msc A)}$  operators associated to $\mc E^\Lambda(\msc A)$ are  defined as follows:   
$$	T_{\mc E^\Lambda(\msc A)}  f (k, \lambda)=\lan f,  L_\lambda \varphi_k\ran, \ \mbox{and} \  \lan T_{\mc E^\Lambda(\msc A)}^* g, f\ran= \sum_{k \in I}\int_{G} g(\lambda, k)\lan L_\lambda \varphi_k, f \ran  \  d{\lambda},$$ for all $k\in I, \lambda \in \Lambda$,  
$g \in L^2 (I \times \Lambda),  \, f \in L^2 (G),$  respectively. The reproducing formula  (\ref{rformula})    for all $f \in \mc S^\Lambda(\msc A)$ can be expressed   $$T_{\mc E^\Lambda(\msc A)}^*T_{\mc E^\Lambda(\msc A')}\big|_{\mc S^\Lambda(\msc A)}  =I_{\mc S^\Lambda(\msc A)}$$ in the form of operators, where $I_{\mc S^\Lambda(\msc A)}$  is the identity operator restricted on the $\mc S^\Lambda(\msc A)$.  Such duals   have been investigated for the discrete frames   by many authors including  \cite{heil2009duals,hemmat2007uniqueness}  in the  Euclidean setup.

	%
	%
	%
	%

We state our main results of this section which characterize alternate (oblique) duals and type-I (type-II) duals in the nilpotent Lie group setup. Our characterization is based on the range function technique for SI/Z Lie group associated with the representation $\tilde{\pi}_\alpha$.  A  range function is a  measurable map from $\mathbb T^r$ to the collection of closed subspaces of $ \ell^2(\mathbb Z^r, \HS(L^2(\mathbb R^d)))$. 

\begin{thm}\label{T:dualnondiscrete}  
	For a.e. $\alpha \in \mathbb T^r$, we consider the range function   
	\begin{align}\label{range}
		J_{\msc A}(\alpha)={\overline{\Span}} \{\tilde{\pi}_{\alpha}(\lambda_1)\msc F\varphi(\alpha):\varphi\in \msc A, \lambda_1\in \Lambda_1\} \subseteq \Ell ,
	\end{align}
	associated with the representation $\tilde{\pi}_{\alpha}$.  Then   $\mc E^\Lambda(\msc A')$ is an alternate (oblique) dual of  $\mc E^\Lambda(\msc A)$   if and only if 
	for a.e. $\alpha\in \mathbb  T^r$, the system $\{\tilde{\pi}_{\alpha}(\lambda_1)\msc F\psi(\alpha):\psi \in \msc A', \lambda\in \Lambda_1\}$ is an alternate (oblique) dual of the continuous  frame $ \{\tilde{\pi}_{\alpha}(\lambda_1)\msc F\varphi(\alpha):\varphi\in \msc A, \lambda_1\in \Lambda_1\}$ for   $J_{\msc A}(\alpha)$, i.e., 	for all $h\in J_{\msc A}(\alpha)$, 
	$$h=\sum_{k \in I}\int_{ \Lambda_1}\langle h, \tilde{\pi}_{\alpha}(\lambda_1)\msc F\psi_k(\alpha)\rangle   \tilde{\pi}_{\alpha}(\lambda_1)\msc F\varphi_k(\alpha) d\lambda_1.$$ 
\end{thm}
%

It can be noted further that the below characterizations for type-I and type-II duals behaves similar to the Theorem \ref{T:dualnondiscrete} of  the alternate (oblique) duals while type-I and type-II duals are the particular cases of the alternate (oblique) duals.
\begin{thm}\label{T:type I}
$\mc E^\Lambda(\msc A')$  is a type-I (type-II) dual of	$\mc E^\Lambda(\msc A)$  if and only if  
for a.e. $\alpha\in \mathbb  T^r$,  the system  $\{\tilde{\pi}_{\alpha}(\lambda_1)\msc F\psi(\alpha):\psi\in \msc A', \lambda\in \Lambda_1\}$ is a type-I (type-II) dual of the continuous  frame $\{\tilde{\pi}_{\alpha}(\lambda_1)\msc F\varphi(\alpha):\varphi\in \msc A, \lambda_1\in \Lambda_1\}$ for   $J_{\msc A}(\alpha)$, where   the range function  $J_{\msc A}(\alpha)$ is defined by  (\ref{range}) for a.e. $\alpha \in \mathbb T^r$.  
\end{thm}
 \begin{remark}
			For the Heisenberg group $\mathbb H^d$ identified with $\mathbb R^d\times \mathbb R^d\times\mathbb R$ consider   $$\Lambda_1=\mathbb R^d\times \mathbb R^d\times \{0\} \  \mbox{and} \  \Lambda_0=\{(0,0)\}\times\mathbb Z,$$ then the set $\Lambda=\Lambda_1\Lambda_0$ can be identified with $\mathbb R^d\times \mathbb R^d\times \mathbb Z$. 
			For a.e. $\alpha \in \mathbb T$, we consider the range function 
			 $$J_{\msc A}(\alpha)=\overline{\Span} \{\tilde{\pi}_{\alpha}(\lambda_1)\msc F\varphi(\alpha):\varphi\in \msc A, \lambda_1\in \Lambda_1\} \subseteq {{\ell^2(\mathbb Z,\mc B_2(L^2(\mathbb R^d)))}}$$
		associated with the representation $\tilde{\pi}_{\alpha}$, which is defined by, $$\tilde{\pi}_{\alpha}(\lambda_1)z(m):={\pi}_{\alpha+m}(\lambda_1)\degree z(m), 
		 \mbox{where} \  (z(m))\in  {{\ell^2(\mathbb Z,\mc B_2(L^2(\mathbb R^d)))}}.$$ 
		  For $\sigma \in \mathbb R^*=\mathbb R\backslash \{0\}$ and $u=(x,y,z)\in \mathbb H^d$, the Schr\"{o}dinger representations $\pi_{\sigma}(u)$ on  $L^2(\mathbb R^d)$  is given below for $f \in L^2(\mathbb R^d),$
		 $$\pi_{\sigma}(u)f(x')=\pi_{\sigma}(x, y, z)f(x')=e^{2\pi i \sigma z}e^{-2\pi i \sigma y. x'}f(x'-x),  \    x, y, x' \in \mathbb R^d \ \mbox{and} \ z \in \mathbb R.$$  For   $ \varphi \in L^1(\mathbb H^d)\cap L^2(\mathbb H^d)$, the Fourier transform  is defined by:
		 $$\mc F\varphi(\sigma)=\int_{\mathbb H^d}\varphi(x)\pi_{\sigma}(x)dx, \  \sigma \in \mathbb R^*,$$ and
		   the fiberization map  
		 $\msc F : L^2(\mathbb H^d) \rightarrow L^2(\mathbb T; \ell^2(\mathbb Z,\HS(L^2(\mathbb R^d))))$ is given by 
		 $$\msc F(\varphi)(\alpha)(m) = |\alpha+ m|^{\frac{d}{2}}\mc F\varphi (\alpha + m).$$
	For $\msc A\subset L^2(\mathbb H^d)$, the $\Lambda$-generated system $\mc E^\Lambda(\msc A)$ will be of the form $$\mc E^\Lambda(\msc A)=\{L_{\lambda_1\lambda_0}\varphi : \varphi \in \msc A, \lambda_1\in \mathbb R^d\times \mathbb R^d\times \{0\} , \lambda_0 \in \{(0,0)\}\times\mathbb Z\}.$$	Then we can  state  Theorems \ref{T:dualnondiscrete} and  \ref{T:type I}   for the continuous setup.\\
Similarly the results can be developed for $\Lambda_1$ of the form $\Lambda_1=
	\Gamma_1\times \Gamma_2\times \{0\}$ and $\Lambda_0=\{(0,0)\times m \mathbb Z\}$, where $\Gamma_1$, $\Gamma_2$ are additive subgroups of $\mathbb R^d$ and $m \in \mathbb N$.
\end{remark}

First we proceed by defining the terms $ G_{\lambda_1}^k(\alpha)$ and $H_{\lambda_1}^k(\alpha)$ for each $\lambda_1\in \Lambda_1, k \in I$ and a.e. $\alpha \in \mathbb T^r$ as follows:   
\begin{align}
G_{\lambda_1}^k(\alpha):=\langle \msc Ff(\alpha),\tilde{\pi}_{\alpha}(\lambda_1 )\msc F\varphi_k(\alpha)\rangle \  \mbox{and} \ H_{\lambda_1}^k(\alpha):=\langle \msc Fg(\alpha),\tilde{\pi}_{\alpha}(\lambda_1)\msc F\psi_k(\alpha)\rangle, \ \mbox{for}\ f, g \in L^2(G).
\end{align}

\begin{prop}
For each $k\in I$ and $\lambda_1\in \Lambda_1$, the functions 	   $G_{\lambda_1}^k$ and $H_{\lambda_1}^k$are   in  $L^1({\mathbb T^r})$ and their Fourier transforms  $\widehat{G_{\lambda_1}^k} $ and $\widehat{H_{\lambda_1}^k}$  respectively  are   members of $\ell^2 (\mathbb Z^r)$.
\end{prop}
\begin{proof}
Applying  Cauchy-Schwarz inequality and using the property  of $\msc F$, we have 
{\small\begin{align*}
		\int_{\mathbb T^r}|G_{\lambda_1}^k(\alpha)|d\alpha& 
		\leq \left(	\int_{\mathbb T^r} \sum_{m\in \mathbb 
			Z^r}\|\msc Ff(\alpha)(m)\|^2 d\alpha \right)^{1/2}
		\left(	\int_{\mathbb T^r} \sum_{m\in \mathbb 
			Z^r}\|\msc FL_{\lambda_1}\varphi_k(\alpha)(m)\|^2 d\alpha \right)^{1/2}\\
		&=\|\msc F f\| \|\msc FL_{\lambda_1} \varphi_k\|=\|f\|\|\varphi_k\|<\infty,
\end{align*}}
since $\msc F(L_{\lambda_1}\varphi_k)(\alpha)= e^{2\pi i \langle\alpha,0\rangle_{\mathbb T^r}}\tilde{\pi}_{\alpha}(\lambda_1)\msc F\varphi(\alpha)=\tilde{\pi}_{\alpha}(\lambda_1)\msc F\varphi(\alpha)$ and the left translation $L_{\lambda_1}$ is an isometry. 
Hence $G_{\lambda_1}^k \in L^1 (\mathbb T^r)$.  Similarly,  $H_{\lambda_1}^k \in L^1 (\mathbb T^r)$.
The  Fourier transform  of $G_{\lambda_1}^k $ and $H_{\lambda_1}^k$   at $\lambda_0\in \Lambda_0$ is given by 
$\widehat{G_{\lambda_1}^k }(\lambda_0)=\int_{\mathbb T^r}G_{\lambda_1}^k (\alpha)e^{-2\pi i\langle \alpha, \lambda_0\rangle}d\alpha,  \ \mbox{and} \ \widehat{H_{\lambda_1}^k} (\lambda_0)=\int_{\mathbb T^r}H_{\lambda_1}^k (\alpha)e^{-2\pi i\langle \alpha, \lambda_0\rangle}d\alpha .$
Then the sequence  $\{\widehat{G_{\lambda_1}^k }(\lambda_0)\}_{\lambda_0\in \Lambda_0}\in \ell^2(\mathbb Z^r)$,  follows by  observing the Bessel property of $\mc E^\Lambda(\msc A)$  and   properties of $\msc F$  (\ref{T:fiberization})  in  the following calculations
{\small\begin{align*}
		\infty>&\sum_{k \in I}\int_{\Lambda}\left | \langle f, L_{\lambda}\varphi_k\right \rangle |^2 d\lambda
		=\sum_{k \in I}\int_{ \Lambda_1}\sum_{ \lambda_0 \in \Lambda_0}\left | \int_{\mathbb T^r}\langle \msc Ff(\alpha), \tilde{\pi}_{\alpha}(\lambda_1)\msc F\varphi_k(\alpha) \rangle e^{-2\pi i\langle \alpha,\lambda_0\rangle} d\alpha\right|^2 d\lambda_1 \\
		&=\sum_{k \in I}\int_{ \Lambda_1}\sum_{ \lambda_0 \in \Lambda_0}|\widehat{G_{\lambda_1}^k}(\lambda_0)|^2d\lambda_1.
\end{align*}}
Similarly, we have $\{\widehat{H_{\lambda_1}^k}(\lambda_0)\}_{\lambda_0\in\Lambda_0}\in \ell^2 (\mathbb Z^r)$.  
\end{proof}

\begin{prop}\label{P: Bessel-charecteriazation}  For all $f,g \in L^2(G)$, we have 
\begin{align*}
	\sum_{k\in I}\int_{\Lambda}\left \langle f, L_{\lambda}\varphi_k\right \rangle\langle L_{\lambda}\psi_k,g\rangle d\lambda
	=\sum_{k\in I}\int_{\Lambda_1}\int_{\mathbb T^r} G_{\lambda_1}^k(\alpha)\ol{H_{\lambda_1}^k(\alpha)}d\alpha d\lambda_1,
\end{align*}
where $ \mc E^\Lambda(\msc A)$  and $\mc  E^\Lambda(\msc A')$ are Bessel systems. 
\end{prop}
\begin{proof}
By using the  properties of  the  map $\msc F$ , we have
{\small\begin{align*}
		\sum_{k\in I}\int_{\Lambda}\left \langle f, L_{\lambda}\varphi_k\right \rangle\langle L_{\lambda}\psi_k,g\rangle d\lambda
		&=	\sum_{k\in  I}\int_{\Lambda}\left \langle \msc Ff, \msc FL_{\lambda}\varphi_k\right \rangle\langle \msc FL_{\lambda}\psi_k,\msc Fg\rangle d\lambda \\
		&=\sum_{k\in I}\int_{\Lambda}\left(\int_{\mathbb T^r}\langle \msc Ff(\alpha),\msc FL_\lambda\varphi_k(\alpha)\rangle d\alpha \right)\left(\int_{\mathbb T^r}\langle \msc FL_\lambda\psi_k(\alpha),\msc Fg(\alpha)\rangle d\alpha\right) d\lambda\\
		&=\sum_{k \in I}\int_{ \Lambda}\int_{\mathbb T^r}\langle \msc Ff(\alpha),\tilde{\pi}_{\alpha}(\lambda)\msc F\varphi_k(\alpha)\rangle d\alpha\int_{\mathbb T^r}\langle \tilde{\pi}_{\alpha}(\lambda)\msc F\psi_k(\alpha),\msc Fg(\alpha)\rangle d\alpha d\lambda .
\end{align*}}
Writing $\lambda=\lambda_1 \lambda_0,$ where $\lambda_1\in \Lambda_1, \lambda_0\in \Lambda_0,$  we get  $ \tilde{\pi}_{\alpha}(\lambda_1 \lambda_0)=e^{2\pi i\langle \alpha,\lambda_0\rangle}\tilde{\pi}_{\alpha}(\lambda_1)$, and hence the above  expression can be written  as
{\small
	\begin{align*}
		\sum_{k\in I}\int_{\Lambda} & \left \langle f, L_{\lambda}\varphi_k\right \rangle    \langle L_{\lambda}\psi_k,g\rangle d\lambda
		=\sum_{k \in I}\int_{ \Lambda_1} \sum_{ \lambda_0 \in \Lambda_0}\left(\int_{\mathbb T^r}\langle \msc Ff(\alpha),\tilde{\pi}_{\alpha}(\lambda_1 \lambda_0)\msc F\varphi_k(\alpha)\rangle d\alpha\right)\left(\int_{\mathbb T^r}\langle \tilde{\pi}_{\alpha}(\lambda_1\lambda_0)\msc F\psi_k(\alpha), \msc Fg(\alpha)\rangle d\alpha \right)d\lambda_1\\
		&=\sum_{ k \in I}\int_{ \Lambda_1}\sum_{ \lambda_0 \in \Lambda_0}\left(\int_{\mathbb T^r}e^{-2\pi i\langle \alpha,\lambda_0\rangle}\langle \msc Ff(\alpha),\tilde{\pi}_{\alpha}(\lambda_1)\msc F\varphi_k(\alpha)\rangle d\alpha\right)\left(\int_{\mathbb T^r}e^{2\pi i\langle \alpha,\lambda_0\rangle}\langle \tilde{\pi}_{\alpha}(\lambda_1)\msc F\psi_k(\alpha),\msc Fg(\alpha)\rangle d\alpha\right) d\lambda_1\\
		&=\sum_{k \in I}\int_{ \Lambda_1} \sum_{ \lambda_0 \in \Lambda_0}\left(\int_{\mathbb T^r}e^{-2\pi i\langle \alpha,\lambda_0\rangle} G_{\lambda_1}^k(\alpha) d\alpha\right)\left(\int_{\mathbb T^r}e^{2\pi i\langle \alpha,\lambda_0\rangle}\ol{H_{\lambda_1}^k(\alpha)} d\alpha \right)d\lambda_1
		\\
		&=\sum_{k\in I}\int_{ \Lambda_1}\sum_{ \lambda_0 \in \Lambda_0}\ \widehat{G_{\lambda_1}^k}(\lambda_0)\ol{\widehat{H_{\lambda_1}^k(\lambda_0)}}d\lambda_1 =\sum_{k\in I}\int_{\Lambda_1}\langle \widehat{G_{\lambda_1}^k},\widehat{H_{\lambda_1}^k}\rangle_{\ell^2(\mathbb Z^r)} d\lambda_1 \\
		&=\sum_{ k \in I}\int_{ \Lambda_1}\langle {G_{\lambda_1}^k},{H_{\lambda_1}^k} \rangle d\lambda_1 =\sum_{k \in I}\int_{ \Lambda_1}\int_{\mathbb T^r} {G_{\lambda_1}^k(\alpha)}\overline{{{H_{\lambda_1}^k(\alpha)}}}d\alpha d\lambda_1.
	\end{align*}
}
Hence the result follows. 
\end{proof}

\begin{proof}[Proof of Theorem \ref{T:dualnondiscrete}]  We prove the result for alternate duals and proceeding by a similar manner we can conclude for  oblique duals.  For $f, g\in \mc S^\Lambda(\msc A)$, we have
{\small	\begin{align}\label{eq1-Thm4.1}
		\sum_{k \in I}\int_{ \Lambda}\langle f,L_\lambda\varphi_k\rangle{\langle L_\lambda\psi_k,g\rangle}d\lambda &=\sum_{k\in I}\int_{\Lambda_1}\int_{\mathbb T^r} G_{\lambda_1}^k(\alpha)\ol{H_{\lambda_1}^k(\alpha)}d\alpha d\lambda_1\\
		&=\sum_{k \in I}\int_{ \Lambda_1}\int_{\mathbb T^r} \langle \msc Ff(\alpha),\tilde{\pi}_{\alpha}(\lambda_1)\msc F\varphi_k(\alpha)\rangle{\langle \tilde{\pi}_{\alpha}(\lambda_1)\msc F \psi_k(\alpha), \msc Fg(\alpha)\rangle}d\alpha d\lambda_1  \notag,
\end{align}}
due to the Proposition \ref{P: Bessel-charecteriazation}.  Assume  the system  $\{\tilde{\pi}_{\alpha}(\lambda_1)\msc F\psi_k(\alpha):k \in I, \lambda_1\in \Lambda_1\}$ is an  alternate dual of  the  continuous frame $\{\tilde{\pi}_{\alpha}(\lambda_1)\msc F\varphi_k(\alpha):k \in I,\lambda_1\in \Lambda_1\}$ for  a.e. $\alpha \in \mathbb T^r$, i.e., 
$$\sum_{k \in I}\int_{ \Lambda_1}\langle a,\tilde{\pi}_{\lambda_1}(\alpha) \msc F\varphi_k (\alpha)\rangle\langle \tilde{\pi}_{\lambda_1}(\alpha)  \msc F\psi_k(\alpha),b\rangle d\lambda_1 =\langle a,b\rangle,  \  \mbox{for all} \ a, b\in J_{\msc A}(\alpha).
$$
Employing (\ref{eq1-Thm4.1}),  we obtain 
$\sum_{k \in I}\int_{ \Lambda}\langle f,L_\lambda\varphi_k\rangle{\langle L_\lambda\psi_k,g\rangle}d\lambda 
=\int_{\mathbb T^r} \langle \msc Ff(\alpha), \msc Fg(\alpha)\rangle d\alpha =\langle f,g \rangle,$
since $f\in \mc S^\Lambda(\msc A)$  implies   $\msc Ff(\alpha) \in J_{\msc A}(\alpha)$ for a.e. $\alpha \in \mathbb T^r$. Therefore, 
$\mc E^\Lambda(\msc A')$ is an alternate dual of $\mc E^\Lambda(\msc A)$.

Conversely, assume that  $\mc E^\Lambda(\msc A')$ is an alternate dual of $\mc E^\Lambda(\msc A)$, i.e.,   $\sum_{k \in I}\int_{\Lambda}\left \langle f, L_{\lambda}\varphi_k\right \rangle\langle L_{\lambda}\psi_k,g\rangle d\lambda
=\langle f, g\rangle,$ for all $f, g \in \mc S^\Lambda(\msc A)$. Now using    (\ref{eq1-Thm4.1}), we obtain  
{\small\begin{align}\label{dualconverse}
		\int_{{\mathbb T^r}}\langle \msc Ff(\alpha), \msc Fg(\alpha)\rangle d\alpha
		&=	<f, g>=\sum_{k\in I}\int_{\Lambda}\left \langle f, L_{\lambda}\varphi_k\right \rangle\langle L_{\lambda}\psi_k,g\rangle d\lambda
		\\
		&=\sum_{  k \in I}\int_{ \Lambda_1}\int_{\mathbb T^r} \langle \msc Ff(\alpha),\tilde{\pi}_{\alpha}(\lambda_1)\msc F\varphi_k(\alpha)\rangle{\langle \tilde{\pi}_{\alpha}(\lambda_1)\msc F\psi_k(\alpha),\msc Fg(\alpha)\rangle}d\alpha d\lambda_1.\nonumber
\end{align}}
Then  the expression
$\sum_{ k \in I}\int_{ \Lambda_1} \langle \msc Ff(\alpha),\tilde{\pi}_{\alpha}(\lambda_1)\msc F\varphi_k(\alpha)\rangle{\langle  \tilde{\pi}_{\alpha}(\lambda_1)\msc F\psi_k(\alpha),\msc Fg(\alpha)\rangle} d\lambda_1$ is equal to $\langle \msc Ff(\alpha), \msc Fg(\alpha)\rangle, $ for a.e. $\alpha\in \mathbb T^r.$ Suppose this does not hold. Then there exists a measurable set $\msc D\ss \mathbb T^r$ with positive measure such that  it is not equal for a.e. $\alpha \in \msc D$. 

Let  $\{x_n\}_{n\in \mathbb N}$ be  a  countable dense subset of  $\ell^2(\mathbb Z^r, \HS(L^2(\mathbb R^d)))$ and let $P_{J_{\msc A}}(\alpha)$ be  an orthogonal projection on $J_{\msc A}(\alpha)$. Clearly $\{P_{J_{\msc A}}(\alpha)x_n\}_{n\in \mathbb N}$ is dense in $J_{\msc A}(\alpha)$. For each $i,j\in \mathbb N$, we  consider the set 
{\small	\begin{align*}
		S_{i, j}= \Big\{\alpha \in \mathbb T^r: \rho_{i, j}(\alpha) :=  \sum_{k \in I}\int_{ \Lambda_1} \langle P_{J_{\msc A}}(\alpha)x_i ,\tilde{\pi}_{\alpha}(\lambda_1)\msc F\varphi_k(\alpha)\rangle{\langle \tilde{\pi}_{\alpha}(\lambda_1) \msc F\psi_k(\alpha),P_{J_{\msc A}}(\alpha)x_j\rangle} d\lambda_1-\langle  P_{J_{\msc A}}(\alpha)x_i, P_{J_{\msc A}}(\alpha)x_j\rangle\neq \{0\} \Big\}.
\end{align*}}
Then 	there exist $i_0, j_0\in \mathbb N$ such that the set $E:=S_{i_0, j_0}\cap \msc D$ is of positive measure, and hence one of the  sets, viz.,  
$E_1=\{\alpha \in \mathbb T^r: \mbox{Re}(\rho_{i_0,j_0})>0\}$,   	$E_2=\{\alpha \in \mathbb T^r: \mbox{Re}(\rho_{i_0,j_0})<0\}$, 	 
$E_3=\{\alpha \in \mathbb T^r: \mbox{Im}(\rho_{i_0,j_0})>0\}$, and  	 $E_4=\{\alpha \in \mathbb T^r: \mbox{Im}(\rho_{i_0,j_0})<0\},$ should have positive measure, assume $E_1$. By choosing  $\msc Ff(\alpha)=\chi_{E_1} P_{J_{\msc A}}(\alpha)x_{j_0}$ and $\msc Fg(\alpha)=\chi_{E_1} P_{J_{\msc A}}(\alpha)x_{i_0}$, we have  $f,g\in \mc S^\Lambda(\msc A)$ and in view of  (\ref{dualconverse}),   we reach on a contradiction that the measure of $\msc D$ is positive. Similarly we get contradictions with respect to the other sets $E_2$,    $E_3$ and  $E_4$.  Hence the result follows for an alternate dual. 
\end{proof}


%

\begin{prop}\label{L:dualtype2} For $k \in I$ and $\lambda_1 \in \Lambda_1$, let us assume a measurable $\mathbb Z^r$-periodic function $p_{\lambda_1}^k$ satisfying 
$\sum_{k \in I}\int_{ \Lambda_1}\int_{\mathbb T^r} |p_{\lambda_1}^k(\alpha)|^2d\alpha d\lambda_1<\infty$.  Then, for a Bessel system $\mc E^{\Lambda} (\msc A)$, the following are equivalent:  
$$
(i)  \ \sum_{k \in I}\int_{ \Lambda_1}\int_{\mathbb T^r}G_{\lambda_1}^k(\alpha) p_{\lambda_1}^k(\alpha)d\alpha d\lambda_1 =0. \qquad 
(ii)	\  \sum_{k \in I}\int_{ \Lambda_1}G_{\lambda_1}^k(\alpha) p_{\lambda_1}^k(\alpha)d\lambda_1=0, \  \mbox{for a.e.} \ \alpha \in \mathbb T^r.
$$
\end{prop}
\begin{proof} Assume (i).  For (ii) we proceed by the contradiction.  	If (ii) does not hold true,  there exists  a measurable set $\msc D\ss \mathbb T^r$   of  positive measure such that $\sum_{k\in I}\int_{ \Lambda_1}G_{\lambda_1}^k(\alpha) p_{\lambda_1}^k(\alpha)d\lambda_1\neq 0$, for  a.e.  $\alpha \in \msc D$. Let $\{x_i\}_{i\in \mathbb N}$ be a countable dense subset of $\ell^2(\mathbb Z^r, \HS(L^2(\mathbb R^d)))$ and for a.e. $\alpha \in \mathbb T^r$,   let  $P_{J_{\msc A}}(\alpha)$ be an orthogonal projection on $J_{\msc A}(\alpha)$. Then   $\{ P_{J_{\msc A}}(\alpha)x_i\}_{i\in \mathbb N}$ is dense in $J_{\msc A}(\alpha)$ and there exists  an  $i_0\in \mathbb N$ such that $h(\alpha):=\sum_{k \in I}\int_{\Lambda_1}\langle  P_{J_{\msc A}}(\alpha)x_{i_0},\tilde{\pi}_{\alpha}(\lambda_1)\msc F\varphi(\alpha)\rangle p_{\lambda_1}^{k}(\alpha)d\lambda_1\neq 0$ on some measurable set   $Y$    in $\msc D$ having positive measure. Now the proof follows by considering the real and imaginary parts, and by choosing suitable function, the way we did for the proof of Theorem \ref{T:dualnondiscrete} . 
%
\noindent Conversely, we  assume (ii).   The part (i) follows  easily   by  integrating (ii) with respect to  the torus $\mathbb T^r$. 
\end{proof}

\begin{proof}[Proof of Theorem \ref{T:type I}]  We first prove the result for type-I duals and then for type-II duals. We assume   $T_{\msc A(\alpha)}$ and  $T_{\msc A'(\alpha)}$  be   analysis operators associated to the Bessel systems 
$\msc A(\alpha):=\{\tilde{\pi}_{\alpha}(\lambda_1)\msc F\varphi_k(\alpha):k\in I, \lambda_1\in \Lambda_1\}$ and $\msc A'(\alpha):=\{\tilde{\pi}_{\alpha}(\lambda_1)\msc F\psi_k(\alpha):k\in I, \lambda_1\in \Lambda_1\}$, respectively. It's adjoint operators are $T^*_{{\msc A}(\alpha)}$ and $T^*_{{\msc A'}(\alpha)}$.

\noindent\textbf{For Type-I duals:}  	In view of Theorem \ref{T:dualnondiscrete}, it suffices to show       $\mbox{range}\,  T^*_{\mc E^\Lambda(\msc A')} \ss \mbox{range}\,  T^*_{\mc E^\Lambda (\msc A)}$ if and only if $\mbox{range}\,  T^*_{\msc A'(\alpha)} \ss \mbox{range}\,  T^*_{\msc A(\alpha)}$. Equivalently, 
$\left[T^*_{\mc E^\Lambda(\msc A')} (L^2(I\times \Lambda)) \right]\cap \mc S^\Lambda(\msc A') \ss \left[T^*_{\mc E^\Lambda(\msc A)} (L^2(I \times \Lambda))\right] \cap S^\Lambda(\msc A),$
if and only if for a.e. $\alpha \in \mathbb T^r$, 
$\left[T^*_{\msc A'(\alpha)} (L^2(I\times \Lambda_1))\right] \cap J_{ A'}(\alpha) \ss \left[T^*_{ \msc  A (\alpha)} (L^2(I\times \Lambda_1))\right] \cap J_{\msc A}(\alpha).$  For this it is enough to verify   on the generators $\varphi_k, \psi_k, L_\lambda\varphi_k,$ and $L_\lambda\psi_k$,  for $k \in I$ and $\lambda \in \Lambda$. Then   the result follows   just by observing:    $L_\lambda\psi_k\in \mc S^\Lambda(\msc A)$, for  $k\in I$ and $\lambda\in \Lambda$  if and only if  for a.e. $\alpha \in \mathbb T^r$, $\msc FL_{\lambda_1}\psi_{k'} (\alpha) \in J_{\msc A}(\alpha)$, for $k' \in I, \lambda_1 \in \Lambda_1$.

\noindent \textbf{For Type-II duals:} In view of Theorem \ref{T:dualnondiscrete}, it suffices to show    	range $T_{\mc E^\Lambda(\msc A')}\ss \mbox{range} \ T_{\mc E^\Lambda(\msc A)}$ if and only if for a.e. $\alpha\in \mathbb T^r$, range $T_{\msc A'}(\alpha)\ss \mbox{range} \ T_{\msc A(\alpha)}$.  
First  we assume   range $T_{{\msc A'}(\alpha)}\ss \ \mbox{range} \ T_{{\msc A}(\alpha)}$, a.e. $\alpha\in\mathbb T^r$. Then the family  $\{a_{k,\lambda_1}\}_{k \in I, \lambda_1\in \Lambda_1}$ in $L^2(I \times \Lambda_1)$ satisfies for a.e. $\alpha \in \mathbb T^r$.   {\small\begin{align}\label{eq:dualtype2range}
		\sum_{ k \in I}\int_{ \Lambda_1}\langle h, \tilde{\pi}_{\alpha}(\lambda_1) \msc F\varphi_k(\alpha)\rangle \overline{a_{k,{\lambda_1}}}d\lambda_1=0, \forall h\in J_{\msc A}(\alpha)\implies  \sum_{ k\in I}\int_{ \Lambda_1}\langle h, \tilde{\pi}_{\alpha}(\lambda_1)\msc F\psi_k(\alpha)\rangle\overline{a_{k,\lambda_1}}d\lambda_1=0,  \forall h\in J_{\msc A}(\alpha)
\end{align}}
To prove  range $T_{\mc E^\Lambda(\msc A')}\ss \mbox{range} \ T_{\mc E^\Lambda(\msc A)}$, 	we calculate the following for $\{c_{k,\lambda}\}_{k \in I, \lambda\in \Lambda}=\{c_{k,\lambda_0, \lambda_1}\}_{k \in I, \lambda_0\in \Lambda_0, \lambda_1\in \Lambda_1}$ in $L^2(I \times \Lambda)$ as follows:
{\small\begin{align*}
		\sum_{ k \in I}	\int_{\Lambda}\langle f, L_{\lambda}\varphi_k\rangle\overline{c_{k,\lambda}} d\lambda
		&=\sum_{k \in I}\int_{ \Lambda_1}\sum_{ \lambda_0 \in \Lambda_0}\int_{\mathbb T^r}\langle \msc Ff(\alpha), \msc FL_{\lambda_1 \lambda_0}\varphi_k(\alpha)\rangle  \overline{c_{k,\lambda_1,\lambda_0}} d\alpha d\lambda_1\\
		&	=\sum_{k\in I}\int_{ \Lambda_1}\sum_{ \lambda_0 \in \Lambda_0}\int_{\mathbb T^r}\langle \msc Ff(\alpha), \msc FL_{\lambda_1}\varphi_k(\alpha)\rangle e^{-2\pi i \langle\alpha, \lambda_0\rangle}  \overline{c_{k,\lambda_1, \lambda_0}} d\alpha d\lambda_1\\ 
		&	=\sum_{ k \in I}\int_{ \Lambda_1}\int_{\mathbb T^r}\langle \msc Ff(\alpha), \tilde{\pi}_{\alpha}(\lambda_1)\msc F\varphi_k(\alpha)\rangle \overline{p_{\lambda_1}^k(\alpha)}  d\alpha d\lambda_1,  
\end{align*}}
where $p_{\lambda_1}^k(\alpha)=\sum_{\lambda_0\in \Lambda_{0}}{c_{k,\lambda_1, \lambda_0}}e^{2\pi i\langle \alpha,\lambda_0\rangle}$ satisfies   $$\sum_{k \in I}\int_{\Lambda_1}\int_{\mathbb T^r}|p_{\lambda_1}^k(\alpha)|^2d\alpha d\lambda_1  =\sum_{k \in I}\int_{ \Lambda_1}\sum_{\lambda_0\in \Lambda_0}|c_{k,\lambda_1,\lambda_0}|^2d\lambda_1<\infty.$$
Similarly, we can obtain 
$\sum_{ k \in I}\int_{\Lambda}\langle f, L_{\lambda}\psi_k\rangle\overline{c_{k,\lambda}}d\lambda=\sum_{ k\in I}\int_{ \Lambda_1}\int_{\mathbb T^r}\langle \msc Ff(\alpha),\tilde{\pi}_{\alpha}(\lambda_1) \msc F\psi_k)(\alpha)\rangle \overline{p_{\lambda_1}^k(\alpha)}  d\alpha d\lambda_1$. 

By assuming $\sum_{k \in I}	\int_{ \Lambda}\langle f, L_{\lambda}\varphi_k\rangle\overline{c_{k,\lambda}}d\lambda =0$ for all $f\in \mc S^\Lambda(\msc A)$,  and applying Proposition \ref{L:dualtype2} we get  for a.e. $\alpha\in \mathbb T^r $,
\begin{align}\label{eq:dual2i}
	\sum_{k \in I}\int_{\Lambda_1}\langle \msc Ff(\alpha),\tilde{\pi}_{\alpha}(\lambda_1) \msc F\varphi_k(\alpha)\rangle \overline{p_{\lambda_1}^k(\alpha)} d\lambda_1=0 \implies  \sum_{ k \in I}\int_{ \Lambda_1}\langle \msc Ff(\alpha), \tilde{\pi}_{\alpha}(\lambda_1) \msc F\psi_k(\alpha)\rangle \overline{p_{\lambda_1}^k(\alpha)} d\lambda_1 =0,
\end{align} 
from (\ref{eq:dualtype2range}). 
Therefore, we get 	$\sum_{ k \in I}\int_{\lambda\in \Lambda}\langle f, L_{\lambda}\psi_k\rangle\overline{c_{k,\lambda}}d\lambda=0$, 
for $f\in \mc S^\Lambda(\msc A)$. Thus, range $T_{\mc E^\Lambda(\msc A')}\ss$ range $ T_{\mc E^\Lambda(\msc A)}$. Conversely, we assume,   range $T_{\mc E^\Lambda(\msc A')}\ss$ range $ T_{\mc E^\Lambda(\msc A)}$.  For   range $T_{{\msc A'}(\alpha)}\ss \ \mbox{range} \ T_{{\msc A}(\alpha)}$, a.e. $\alpha\in\mathbb T^r$, we can proceed with the help of  (\ref{eq:dual2i})  and choosing   $ p_{\lambda_1}^k(\alpha)=c_{k, \lambda_1}$ for all $\alpha \in \mathbb T^r$. \end{proof}

\section{ Reproducing formulas by the action of     discrete translations}\label{S:Rieszbasis}

In this section  we assume  a discrete  subset $\Lambda_1\ss\exp\mathbb RX_{r+1}\dots \exp \mathbb RX_{n}$  and the  integer lattice $\Lambda_0=\exp \mathbb ZX_1\dots\exp \mathbb ZX_r $ in $G$.  Our aim is to obtain results related to reproducing formula of a Bessel family  in terms of the bracket map for a  $\Lambda$-translation generated system having   biorthogonal property, where $\Lambda=\Lambda_1\Lambda_0$. In the sequel we use the operator $[\cdot, \cdot]: L^2(G) \times L^2(G) \ra  L^1(\mathbb T^r)$, known as   \textit{bracket map}, defined as follows for a.e. $\alpha \in \mathbb T^r, \varphi,\psi \in L^2(G)$:
 $$[\varphi, \psi](\alpha) :=\langle \msc F\varphi(\alpha), \msc F\psi(\alpha)\rangle_{\Ell}=\sum_{m \in \mathbb Z^r}\langle \mc F\varphi(\alpha+m),\mc F\psi(\alpha+m)\rangle |{\bf Pf}(\alpha +m)|,  $$
 to address the results related to the reproducing formulas \cite{bar2014bracket}.
  Recall the   $\Lambda$-translation generated system  $\mc E^{\Lambda} (\varphi)=\{L_\lambda \varphi :    \lambda \in  \Lambda\} $ and its associated    $\Lambda$-translation invariant  space  $\mc S^{\Lambda} (\varphi)=\ol{\mbox{span}} \ \mc E^{\Lambda}(\varphi)$  in $L^2 (G)$ from  (\ref{TIsystem}), where $\varphi \in L^2 (G)$. Then for the reproducing formulas associated to the  system $\mc E^{\Lambda} (\varphi)$, we proceed by considering biorthogonal systems generated by the  discrete translations.   $\mc E^\Lambda(\varphi)$ and $\mc E^\Lambda(\psi)$ are  \textit{biorthogonal} if    $\langle \varphi, L_\lambda  \psi\rangle=\delta_{\lambda, 0}$ for all $\lambda\in \Lambda$.  We obtain  a necessary and sufficient  condition  for the biorthogonality and orthogonality of translation generated systems   in terms of the bracket map.  
 \begin{prop}\label{P:bior}
 	Let $\varphi,\psi\in L^2(G)$ be non-zero functions. Then    the following hold:
 	\begin{enumerate}
 		\item[(i)] $\mc E^\Lambda(\varphi)$ and  $\mc E^\Lambda(\psi)$ are     biorthogonal if and only if   
 		$[\varphi,  L_{\lambda_1}\psi](\alpha)=\delta_{\lambda_1,0},  \   \mbox{for all} \ \lambda_1\in \Lambda_1, \  \mbox{a.e.} \  \alpha\in \mathbb T^r.$
 		
 		\item [(ii)]  The subspace generated by $\mc E^\Lambda(\varphi)$ is  orthogonal  to  $\mc E^\Lambda(\psi)$   if and only if 
 		\begin{equation*}
 			[\varphi, L_{\lambda_1}\psi](\alpha)=0, \  \mbox{for all} \ \lambda_1\in \Lambda_1,  \ \mbox{ a.e.} \  \alpha\in \Omega_{\varphi}:=\{\alpha \in \mathbb T^r: [\varphi,\varphi](\alpha)\neq 0\}.
 		\end{equation*}
 		In particular,  $\mc E^\Lambda(\varphi)$ is an orthogonal system of functions if and only if the orthogonality condition  
 		\begin{equation}\label{OGC}
 			\mathcal{O}_\varphi: \quad 	[\varphi, L_{\lambda_1}\varphi](\alpha)=0,  \  \mbox{for all} \ \lambda_1\in \Lambda_1\backslash \{0\}, \  \mbox{ a.e.} \  \alpha\in \Omega_{\varphi}, \quad \mbox{holds}.  
 		\end{equation}
 		
 	\end{enumerate}
 \end{prop} 
The proof of  Proposition \ref{P:bior}   can be realized on the same technique followed in   \cite{arati2019orthonormality, bar2014bracket} for the Heisenberg group.   In the wake of Proposition \ref{P:bior},  we observe that the biorthogonality (or, orthogonality) of $\Lambda$-translation generated systems is equivalent to the corresponding biorthogonality (or, orthogonality)  of $\Lambda_0$-translation generated systems.

 Now we state our main result for a $\Lambda$-translation generated  system in $L^2 (G)$ to form reproducing formula. Unlike the case of the Euclidean setup, we observe that a necessary condition is involved related to the orthogonality of $\Lambda$-translation generated system of functions. 
  
 \begin{thm}\label{T:Bio1}	Let $\varphi, \psi \in L^2(G)$ be  two functions  such that it satisfies the orthogonality  conditions $\mathcal O_\varphi$ and    	$\mathcal O_\psi$ mentioned in (\ref{OGC}).  If $\mc E^\Lambda(\psi)$ is biorthogonal to $\mc E^\Lambda(\varphi)$, the  following reproducing  formula holds true:
 	$$f=\sum_{  \lambda \in \Lambda}\langle f, L_{\lambda}\psi\rangle L_{\lambda}\varphi, \ \mbox{for all} \  f\in \mc S^\Lambda(\varphi).$$  
 \end{thm}
Next, we discuss reproducing formula for a $\Lambda$-translation generated system, and we find an easily verifiable condition to satisfy the reproducing formula.

\begin{thm}\label{dual-bracket}  
	Let  $\varphi, \psi \in  L^2(G)$  be such that    $\varphi$ and $\psi$ satisfy  the orthogonality  conditions $\mathcal O_\varphi$ and    	$\mathcal O_\psi$ mentioned in (\ref{OGC}), respectively, then following are equivalent:
	
	\quad \quad  (i) \ $f=\sum_{  \lambda \in \Lambda}\langle f, L_{\lambda}\psi\rangle L_{\lambda}\varphi$,  for all $f\in \mc S^\Lambda(\varphi).$ \qquad
	(ii) \ $[  \varphi,   \psi](\alpha)
	=1,  \  a.e.  \ \alpha \in   \Omega_{\varphi}. $
\end{thm}

As a consequence of Theorem \ref{dual-bracket} for the Heisenberg group $\mathbb H^d$, next we discuss a reproducing formula associated with the orthonormal Gabor systems of $L^2(\mathbb R^d)$. 
For $y\in \mathbb R^*$, we  define functions $v_y$ and $w_y$, from \cite{bar2014bracket},  such that $v_y(x)=|y|^{d/2}v(yx)$ and $w_y(x)=|y|^{d/2}w(yx)$,  $x \in \mathbb R^d$, where $v,w\in L^2(\mathbb R^d)$  with  $\|v\|=1,\|w\|=1$.   Corresponding to $v_y$ and $w_y$, we consider the rank one projection operators $\mc P_y$ and $\mc Q_y$ defined as follows:
 $$\mc P_y=v_y\otimes v_y: L^2(\mathbb R^d)\ra L^2(\mathbb R^d) \
 \mbox{ by,}\  f\ra \langle f, v_y\rangle v_y \ \mbox{and} \ \mc Q_y=w_y\otimes w_y: L^2(\mathbb R^d)\ra L^2(\mathbb R^d)\ \mbox{by},\ f\ra \langle f, w_y\rangle w_y.
  $$
 Next for all $t\in (0,1)$,  we define
 $$\msc H_t(y)=\begin{cases}
 \mc P_y\quad  \mbox{for} \ y \in (t, 1]\\
 0\ \quad \  \mbox{otherwise}
 \end{cases}\hspace{2cm}
 \mbox{and}
 \hspace{2cm}
  \msc G_t(y)=\begin{cases}
 \mc Q_y \quad  \mbox{for} \ y \in (t, 1]\\
 0 \ \ \quad \mbox{otherwise}.
 \end{cases}$$ 
  Then  $\msc H_t,  \msc G_t\in L^2(\mathbb R^*, \HS(L^2(\mathbb R^d))|\lambda|^d d\lambda)$  since  
  $\|\msc H_{t}(y)\|_{\HS}=\begin{cases}
 1 \  \quad  \mbox{for} \ y \in (t, 1],\\
 0 \ \ \quad  \mbox{otherwise}.
 \end{cases}=\|\msc G_{t}(y)\|_{\HS}$,  
 and hence for each $t\in (0,1)$,  $\varphi_t,  \psi_t \in L^2(\mathbb H^d)$, where 
  $\varphi_t = \mathcal F^{-1}\msc H_t \  \mbox{and} \   \psi_t= \mathcal F^{-1}\msc G_t.$

\begin{thm}\label{Application}  Let $A, B\in GL(d,\mathbb R)$ such that  $AB^t\in \mathbb Z$. 	For  $ 0<\alpha\leq 1$, let the Gabor systems $$
	\{\pi_{\alpha}(m, n)v_{\alpha}: (m, n)\in A\mathbb Z^d\times B\mathbb Z^d\}, \ \mbox{and} \ \{\pi_{\alpha}(m, n)w_{\alpha}:(m, n)\in A\mathbb Z^d\times B\mathbb Z^d\}
	$$ 
	be  orthonormal   in $L^2(\mathbb R^d)$. Then  for each   $0<t< \alpha\leq 1$,  the systems $\mc E^\Lambda(\varphi_t)$ and $\mc E^\Lambda(\psi_t)$ satisfy the reproducing formula
	$$
	f=\sum_{  \lambda \in \Lambda}\langle f, L_{\lambda}\psi_t\rangle L_{\lambda}\varphi_t,  \ \forall f\in   \ \mc S^\Lambda(\varphi_t)\
	 \mbox{if and only if}  \quad  |\langle v_{\alpha}, w_{\alpha}\rangle_{L^2(\mathbb R^d)}|=\frac{1}{\alpha^{d/2}},
	 $$
	 where $\Lambda=A\mathbb Z^d\times B\mathbb Z^d \times \mathbb Z$.
\end{thm}

%

To proceed for the proof of Theorems \ref{T:Bio1}, \ref{dual-bracket}, and \ref{Application}, we first note  Proposition \ref{P:bior}   motivates  to decompose the principle translation invariant space $\mc S^\Lambda(\varphi)$ into the orthogonal direct sum of $\Lambda_1$-translates.  
\begin{prop}\label{L:decomposition}
	Let $\varphi\in L^2(G)$ be   such that   it satisfies the orthogonality condition $\mathcal O_\varphi$ mentioned in  (\ref{OGC}).  
	For $\lambda_1\neq \lambda_1' \in \Lambda_1$,    the subspace generated by $\mc E^{\Lambda_0}(L_{\lambda_1}\varphi)$ is orthogonal to $\mc E^{\Lambda_0}(L_{\lambda_1'}\varphi)$ and hence
		$\mc S^\Lambda(\varphi)=\bigoplus_{\lambda_1\in \Lambda_1}\mc S^{\Lambda_0}(L_{\lambda_1}\varphi)$.
		  Moreover, 	$f\in S^{\Lambda}(\varphi)$ if and only if 
		\begin{align*} 
		(\msc Ff)(\alpha)(m) 
		=\sum_{\lambda_1\in\Lambda_1}\mathfrak p_{\lambda_1}(\alpha)\msc FL_{\lambda_1}\varphi( \alpha)(m),  \ \mbox{for a.e.} \  \alpha \in \mathbb T^r, \ m\in \mathbb Z^r,  
		\end{align*} 
		where  $\mathfrak p=\{\mathfrak p_{\lambda_1}\}_{\lambda_1\in \Lambda_1}$ is a member of the weighted space $L^2(\mathbb T^r, [\varphi, \varphi])$. Further  
  there exist unique $\varphi_{\lambda_1},\psi_{\lambda_1} $ in $\mc S^{\Lambda_0}(L_{\lambda_1}\varphi)$ such that 
		$$
		f=\sum_{ \lambda_1 \in \Lambda_1}\varphi_{\lambda_1}, \ g=\sum_{ \lambda_1 \in \Lambda_1}\psi_{\lambda_1}, \ \mbox{and } \  \langle f, g\rangle=\sum_{ \lambda_1 \in \Lambda_1}\langle \varphi_{\lambda_1}, \psi_{\lambda_1}\rangle, \ \mbox{for}\ f, g \in \mc S^\Lambda(\varphi) .$$
		\end{prop}
The above Proposition \ref{L:decomposition} can be proved easily.  

We now a provide a characterization of Bessel family under the orthogonality condition. The sequence $\mc E^\Lambda(\varphi)$ is called Bessel in $\mc S^\Lambda(\varphi)$  if, 	$\sum_{\lambda\in \Lambda}	|\langle f, L_{\lambda}\varphi\rangle|^2 \leq B\|f\|^2 \ , \forall f \in  \mc S^\Lambda(\varphi).$

\begin{prop}\label{Bessel}  Let $\varphi\in L^2(G)$ be   such that   it satisfies the orthogonality condition $\mathcal O_\varphi$ mentioned in  (\ref{OGC}).  Then $\mc E^\Lambda(\varphi)$ is Bessel sequence in $\mc S^\Lambda(\varphi)$. Also, 
	$\mc E^\Lambda(\varphi)$  being a Bessel sequence in $\mc S^\Lambda(\varphi)$ is equivalent to $\mc E^{\Lambda_0}(\varphi)$ being a  Bessel sequence in $\mc S^{\Lambda_0}(\varphi)$. 
	\end{prop}
\begin{proof} 
	The Bessel condition of $\mc E^\Lambda(\varphi)$ is followed from the Parseval equality of  the orthonormal  system $\mc E^\Lambda(\varphi)$.  To show the equivalence, let  $\mc E^{\Lambda_0}(\varphi)$
	 be a Bessel sequence in $\mc S^{\Lambda_0}(\varphi)$ with bound $B$. For $f\in \mc S^\Lambda(\varphi)$,   there exists $\varphi_{\lambda_1}$ in $\mc S^{\Lambda_0}{(L_{\lambda_1}\varphi)}$ such that  $f=\sum_{ \lambda_1 \in \Lambda_1}\varphi_{\lambda_1}$, and hence by using Proposition \ref{L:decomposition}, we obtain 
		\begin{align*}
	\sum_{  \lambda \in \Lambda}|\langle f, L_\lambda \varphi\rangle|^2&=\sum_{\eta_1\in \Lambda_1}\sum_{\eta_0\in \Lambda_0}|\langle \sum_{\lambda_1\in \Lambda_1}\varphi_{\lambda_1}, L_{\eta_1}L_{\eta_0}\varphi\rangle|^2
	=\sum_{\eta_1 \in \Lambda_1}\sum_{  \eta_0\in \Lambda_0}|\langle L_{{\eta_1}^{-1}} \varphi_{\eta_1}, L_{\eta_0}\varphi\rangle |^2\\
	&\leq B \sum_{\eta_1 \in \Lambda_1}\|L_{{\eta_1}^{-1}}\varphi_{\eta_1}\|^2=B\sum_{\eta_1\in \Lambda_1}\|\varphi_{\eta_1}\|^2=B\|f\|^2,
	\end{align*}
	since $L_{\eta_1^{-1}} \varphi_{\eta_1}\in \mc S^{\Lambda_0}(\varphi)$ and $L_{{\eta_1}^{-1}}$ is an isometry. Hence $\mc E^\Lambda(\varphi)$ is also Bessel sequence in $\mc S^\Lambda(\varphi)$.
	
	Conversely, assume $\mc E^\Lambda(\varphi)$ is  a Bessel sequence in $\mc S^\Lambda(\varphi)$.  Using Proposition \ref{L:decomposition} , we have $\langle \varphi, L_{\lambda_1}L_{\lambda_0}\varphi\rangle=0$, for   $\lambda_1\neq 0$ as $\varphi \in \mc S^{\Lambda_0}(\varphi)$, and hence the result follows by noting 
	$$\sum_{ \lambda_0 \in \Lambda_0}|\langle \varphi, L_{\lambda_0}\varphi\rangle|^2=\sum_{  \lambda_1 \in \Lambda_1}\sum_{\lambda_0\in \Lambda_0}|\langle \varphi, L_{\lambda_1}L_{\lambda_0}\varphi\rangle|^2\leq B \|\varphi\|^2.$$ 
\end{proof}
Next, we observe that the orthogonality condition (\ref{OGC}) transfers the nature of the reproducing formula of  $\Lambda$-translation generated systems to the $\Lambda_0$-translation generated systems.     
\begin{prop}\label{OGCreproducing}  
Let $\varphi, \psi\in L^2(G)$ be  two functions such that    $\varphi$ and $\psi$ satisfy  the orthogonality  conditions $\mathcal O_\varphi$ and    	$\mathcal O_\psi$ mentioned in (\ref{OGC}), respectively, then  the following are equivalent: 
\begin{itemize}
 \item[(i)] $\langle f, g\rangle=\sum_{\lambda \in \Lambda} \langle f, L_{\lambda}\psi\rangle \langle L_{\lambda}\varphi, g\rangle, \,  \forall f,g\in \mc S^\Lambda(\varphi).$ \\ 
	\item[(i)]  $\langle f, g\rangle=\sum_{  \lambda_0 \in \Lambda_0}\langle f, L_{\lambda_0}\psi\rangle \langle L_{\lambda_0}\varphi,g\rangle, \, \forall f, g\in \mc S^{\Lambda_0}(\varphi).$
	\end{itemize}
 \end{prop}

\begin{proof} Firstly note that the summations used in (i) and (ii) are well defined in view of Proposition \ref{Bessel}, since   $\mc E^\Lambda(\varphi)$ and $\mc E^\Lambda(\psi)$ are Bessel,   and $\varphi$ and $\psi$ satisfy the orthogonality conditions  $\mc O_\varphi$ and $\mc O_\psi$, respectively. 

Now  assume (ii) holds. By choosing  $\varphi_{\lambda_1}, \psi_{\lambda_1} \in \mc S^{\Lambda_0}{(L_{\lambda_1}\varphi)} $ such that $f=\sum_{ \lambda_1 \in \Lambda_1}\varphi_{\lambda_1}, g=\sum_{ \lambda_1 \in \Lambda_1}\psi_{\lambda_1}$ and using  Proposition (\ref{L:decomposition}), we have  
	\small{\begin{align*}
	\sum_{ \lambda \in \Lambda}\left\langle f, L_\lambda\psi\right\rangle \left \langle L_\lambda\varphi, g\right\rangle
&=\sum_{\eta_1\in \Lambda_1}\sum_{\eta_0\in \Lambda_0} \left\langle \sum_{\lambda_1\in \Lambda_1} \varphi_{\lambda_1}, L_{\eta_1}L_{\eta_0}\psi\right\rangle   \left\langle L_{\eta_1}L_{\eta_0}\varphi, \sum_{\lambda_1\in \Lambda_1}\psi_{\lambda_1}\right\rangle
	\\
&=\sum_{\lambda_1\in \Lambda_1}\sum_{\eta_0\in \Lambda_0} \left\langle  L_{\lambda_1^{-1}}\varphi_{\lambda_1}, L_{\eta_0}\psi\right\rangle   \left\langle L_{\eta_0}\varphi, L_{\lambda_1^{-1}} \psi_{\lambda_1}\right\rangle. 
	\end{align*}}
	Since  $L_{\lambda_1^{-1}} \varphi_{\lambda_1}, \ L_{\lambda_1^{-1}} \psi_{\lambda_1}\in \mc S^{\Lambda_0}(\varphi)$,  we obtain the following in view of the  assumption (ii):
	\begin{align*}
	\sum_{ \lambda \in \Lambda}\left\langle f, L_\lambda\psi\right\rangle \left \langle L_\lambda\varphi, g\right\rangle=	\sum_{\lambda_1\in \Lambda_1} \left\langle  L_{\lambda_1^{-1}}\varphi_{\lambda_1},    L_{\lambda_1^{-1}} \psi_{\lambda_1}\right\rangle
	=\sum_{\lambda_1\in \Lambda_1} \left\langle  \varphi_{\lambda_1},     \psi_{\lambda_1}\right\rangle
	=\langle f, g\rangle.
	\end{align*}
	Thus,  (i) holds. Conversely,  assume  (i) holds.  For $f,g\in \mc S^{\Lambda_0}(\varphi)$, we have     $\langle f, L_{\lambda_1}L_{\lambda_0}\varphi\rangle=0$  and $\langle g, L_{\lambda_1}L_{\lambda_0}\varphi\rangle=0$ for all $\lambda_1\neq 0\in \Lambda_1$,   by Proposition \ref{L:decomposition}. Then the result follows immediately. 
\end{proof}

 \begin{proof}[Proof of Theorem \ref{T:Bio1}] Since $\mc E^\Lambda(\varphi)$ and $\mc E^\Lambda(\psi)$ are  biorthogonal, we have $\langle\varphi, L_{\lambda_0}\psi\rangle=\delta_{\lambda_0, 0}$,  for $\lambda_0 \in \Lambda_0$, and hence 
for $f  \in \mbox{span} \  \mc E^{\Lambda_0}(\varphi)$, we write	$f=\sum_{\lambda_0 \in \Lambda_0'}\langle f,L_{\lambda_0}\psi\rangle L_{\lambda_0}\varphi, $
	 for some finite subset $\Lambda_0'$  of $\Lambda_0.$  Since  the   function $f\mapsto \sum_{\lambda_0 \in \Lambda_0}\langle f,L_{\lambda_0}\psi\rangle L_{\lambda_0}\varphi$ is continuous, the expansion  of $f$   holds  for all   $f \in \mc S^{\Lambda_0}(\varphi)$.  Thus the result follows  by Proposition \ref{OGCreproducing}.  
	 	\end{proof}

\begin{proof}[Proof of  Theorem \ref{dual-bracket}] Let us assume (i).  Equivalently,  $f=\sum_{ \lambda_0 \in \Lambda_0}\langle f, L_{\lambda_0}\psi\rangle L_{\lambda_0}\varphi$ for all $f\in \mc S^{\Lambda_0}(\varphi)$ by  Proposition \ref{OGCreproducing}.  
	Applying the map  $\msc F$ on the both sides   for a.e. $\alpha \in \mathbb T^r$ and using the relation $(\msc F L_{\lambda_0} \varphi)(\alpha)=    e^{-2\pi i \langle\alpha, \lambda_0\rangle}\msc F\varphi(\alpha),$  we have 
{\small	\begin{align*}
	\msc F\left(\sum_{ \lambda_0 \in \Lambda_0}\langle f,L_{\lambda_0} \psi\rangle L_{\lambda_0} \varphi\right)(\alpha)
	&=\sum_{\lambda_0 \in \Lambda_0}\langle f,L_{\lambda_0} \psi\rangle (\msc FL_{\lambda_0} \varphi)(\alpha) 
	= (\msc F\varphi)(\alpha)\sum_{ \lambda_0 \in \Lambda_0}\langle f,L_{\lambda_0} \psi\rangle e^{-2\pi i\langle \alpha,{\lambda_0}\rangle}\\
	&= (\msc F\varphi)(\alpha)\sum_{ \lambda_0 \in \Lambda_0}\int_{{\mathbb T^r}}\langle \msc Ff(\alpha),\msc FL_{\lambda_0} \psi(\alpha)\rangle d\alpha e^{-2\pi i \langle\alpha,{\lambda_0}\rangle}\nonumber\\
	&= (\msc F\varphi)(\alpha)\sum_{ \lambda_0 \in \lambda_0}\int_{{\mathbb T^r}}e^{2\pi i \langle \alpha, \lambda_0\rangle}[f,\psi](\alpha)d\alpha e^{-2\pi i \langle \alpha,\lambda_0\rangle}\nonumber\\
	&	=	[f, \psi](\alpha) \, (\msc F\varphi)(\alpha), \ \mbox{for a.e.}\ \alpha \in \mathbb T^r,  
	\end{align*}}
	and hence we get $
	\msc Ff(\alpha)=	\msc F\left(\sum_{\lambda_0 \in \Lambda_0}\langle f, L_{\lambda_0} \psi \rangle L_{\lambda_0}\varphi \right)(\alpha)
		=(\msc F\varphi)(\alpha)[f,\psi](\alpha)$,  
 for all $ f\in \mc S^{\Lambda_0}(\varphi).$  By  choosing   $f=\varphi $, we have
	$
	(\msc F\varphi) (\alpha)\left(1-[\varphi, \psi](\alpha)\right)=0$,  for a.e.   $\alpha \in \mathbb T^r.$ Therefore, we get  $[\varphi, \psi](\alpha)=1$, for  a.e. $\alpha \in \Omega_{ \varphi}$.

	Conversely,  assume (ii),  i.e., $[\varphi, \psi](\alpha)=1$,  a.e. $\alpha \in \Omega_{\varphi}$.  Then it is enough to show   $f=\sum_{  \lambda_0 \in \Lambda_0}\langle f, L_{\lambda_0}\psi\rangle  L_{\lambda_0}\varphi$, for all $f \in \mc S^{\Lambda_0}(\varphi)$ in view of the  Proposition \ref{OGCreproducing}.  Since    the function $f\mapsto \sum_{\lambda_0\in\Lambda_0}\langle f,L_{\lambda_0}\psi\rangle L_{\lambda_0}\varphi$ from $\mc S^{\Lambda_0}(\varphi)$ to  $L^2 (G)$ is continuous,   it suffices to show the result for $f=L_{\eta}\varphi$,  $\eta \in \Lambda_0$. Therefore the result follows  in view of the calculations 
	\begin{align*}
	\msc F\left(\sum_{\lambda\in\Lambda}\langle L_{\eta}\varphi, L_\lambda \psi \rangle L_{\lambda}\varphi \right)(\alpha) &
	=(\msc F\varphi)(\alpha)[(L_\eta\varphi), {\psi}](\alpha)
	= e^{2\pi i \langle \alpha, \eta\rangle}(\msc F\varphi)(\alpha)[\varphi, \psi](\alpha)\\
	&=\msc F (L_{\eta}\varphi)(\alpha) [\varphi, \psi](\alpha)
	=\msc F (L_{\eta}\varphi)(\alpha).
	\end{align*}
    \end{proof}

%

\begin{proof}[Proof of Theorem \ref{Application}]  Observe that for each $0<\alpha \leq 1$,    $0<t< \alpha\leq 1$ and $\lambda_1 \in \Lambda_1=A \mathbb Z^d \times B \mathbb Z^d$, we have 
		\begin{align*}
[L_{\lambda_1} \varphi_t, \varphi_t](\alpha) &=\sum_{m \in \mathbb Z}\langle \mc FL_{\lambda_1} \varphi_t(\alpha+m), \mc F\varphi_t(\alpha+m)\rangle_{\HS} |\alpha+m|^d\\ &=\sum_{m\in \mathbb Z} \langle \pi_{\alpha+m}(\lambda_1) \msc H_t(\alpha+m), \msc H_{t}(\alpha+m)\rangle_{\HS}|\alpha+m|^d\\
&	=\langle \pi_{\alpha}(\lambda_1) \msc H_t(\alpha), \msc H_t(\alpha)\rangle_{\HS}|\alpha|^d
=\langle \pi_{\alpha}(\lambda_1)v_{\alpha}, v_{\alpha}\rangle_{L^2(\mathbb R^d)}|\alpha|^d, 
	\end{align*}
since   $\pi_{\alpha}(\lambda_1)\mc P_{\alpha}=(\pi_{\alpha}(\lambda_1)\mc P_{\alpha}v_{\alpha})\otimes v_{\alpha}$ and
	\begin{align*}
	\langle \pi_{\alpha}(\lambda_1) \msc H_t(\alpha), \msc H_t(\alpha)\rangle_{\HS}=\int_{\mathbb R^d}\int_{\mathbb R^d} \ol{v_{\alpha}(s)}(\pi_{\alpha}(\lambda_1)v_{\alpha}(w))\ol{v_{\alpha}(w)}v_{\alpha}(s)ds dw=\langle \pi_{\alpha}(\lambda_1)v_{\alpha}, v_{\alpha}\rangle_{L^2(\mathbb R^d)}.
	\end{align*}
	Similarly we can calculate  $[L_{\lambda_1} \psi_t, \psi_t](\alpha) =\langle \pi_{\alpha}(\lambda_1)w_{\alpha}, w_{\alpha}\rangle_{L^2(\mathbb R^d)}|\alpha|^d$.
	Therefore for each $0<t< \alpha\leq 1$, the functions $\varphi_t$ and $\psi_t$ satisfy orthogonality conditions $\mc O_{\varphi_t}$ and $\mc O_{\psi_t}$ (mentioned in (\ref{OGC}) )  due to the orthonormal Gabor systems
	$\{\pi_{\alpha}(\lambda_1)v_{\alpha}:\lambda_1\in \Lambda_1\}$ and  $\{\pi_{\alpha}(\lambda_1)w_{\alpha}:\lambda_1\in \Lambda_1\}$.  
	
	  Further observe that $\mc E^\Lambda(\varphi_t)$ is Bessel since $[\varphi_t, \varphi_t](\alpha)$ is bounded above, follows by noting  
	\begin{align*}
		 [\varphi_t, \varphi_t](\alpha)=\sum_{m \in \mathbb Z} \|\mc F\varphi_t(\alpha+m)\|_{\HS}^2 |\alpha+m|^d=\sum_{m \in \mathbb Z} \|\msc H_t(\alpha+m)\|_{\HS}^2 |\alpha+m|^d=\| \msc H_t(\alpha)\|^2|\alpha|^d=\begin{cases}
	\alpha ^d \  t< \alpha \leq 1,\\
	0 \  \mbox{otherwise}.
	\end{cases}
		\end{align*}
		  In a similar way we can obtain $\mc E^\Lambda(\psi_t)$ to be a  Bessel system.  Next we calculate   $ [\varphi_t, \psi_t](\alpha)$ as follows:  
	\begin{align*}
	 [\varphi_t, \psi_t](\alpha) &=\sum_{m \in \mathbb Z}\langle \mc F\varphi_t(\alpha+m),\mc F\psi_t(\alpha +m)\rangle|\alpha+m|^d
	=\sum_{m \in \mathbb Z}\langle \msc H_t(\alpha+m),\msc G_t(\alpha +m)\rangle|\alpha+m|^d\\
	&=\langle \msc H_t(\alpha), \msc G_t(\alpha)\rangle_{\HS}|\alpha|^d
	=|\alpha|^d	\int_{\mathbb R^d} \int_{\mathbb R^d}\ol{v_{\alpha}(s)} v_{\alpha}(u) \ol{w_{\alpha}(u)} w_{\alpha}(s)ds du\\
	&=|\langle v_{\alpha}, w_{\alpha}\rangle|^2|\alpha|^d.
	\end{align*}	 
	  Hence the result follows from Theorem \ref{dual-bracket}  provided   $ [\varphi_t, \psi_t](\alpha)=1$, for $0<\alpha\leq 1$ and $0<t< \alpha\leq 1$.  
\end{proof}
 
	\providecommand{\href}[2]{#2}

 \end{document}